\documentclass[10pt]{amsart}
\usepackage[T1]{fontenc}
\usepackage[cp1250]{inputenc}

\usepackage{graphicx}
\usepackage{enumitem, caption}
\usepackage[sc]{mathpazo}
\usepackage{amssymb, amsthm, amscd, amsmath}
\usepackage{lpic}

\theoremstyle{plain}
\newtheorem{thm}{Theorem}[section]
\newtheorem{prop}[thm]{Proposition}
\newtheorem{lem}[thm]{Lemma}

\theoremstyle{definition}
\newtheorem{defn}[thm]{Definition}
\newtheorem{exmp}[thm]{Example}
\newtheorem{remark}[thm]{Remark}
\newtheorem{ques}[thm]{Question}

\begin{document}


\title{On a surface singular braid monoid}

\author{Micha\l\ Jab\l onowski}
\address{Institute of Mathematics, Faculty of Mathematics, Physics and Informatics, University of Gda\'nsk, 80-308 Gda\'nsk, Poland.}
\email{michal.jablonowski@gmail.com}

\begin{abstract}

We introduce a monoid corresponding to knotted surfaces in four space, from its hyperbolic splitting represented by marked diagram in braid like form. It has four types of generators: two standard braid generators and two of singular type. Then we state relations on words that follow from topological Yoshikawa moves. As a direct application we will reprove some known theorem about twist-spun knots. We wish then to investigate an index associated to the closure of surface singular braid. Using our relations we will prove that there are exactly six types of knotted surfaces with the index less or equal to two, and there are infinitely many types of surface-knots with index equal to three. Towards the end we will construct a family of classical diagrams such that to unlink them requires at least four Reidemeister III moves.
\end{abstract}

\maketitle

\section{Introduction}
\ 

We introduce a monoid $SSB_m$ corresponding to surface-knots in four space, from its hyperbolic splitting represented by marked diagram in braid like form on $m$ strands. It has four types of generators: two standard $c_i$ and $c_i^{-1}$ braid generators and two noninvertible $a_i$ and $b_i$ of singular type. Then we state 11 relations on words that follow from topological Yoshikawa moves from his paper \cite{Yos94} and other interesting relations.

As a rather direct application we will give algebraic formulae for twist-spun knots and reprove some known theorem of Zeeman and Litherland. We wish then to investigate an index associated to the closure of surface singular braid. Our notion of the index is different from the one introduced by Viro and Kamada (cf. \cite{Kam92}). Using our relations we will prove that there are exactly six types of surface-knots with index less or equal to two, and there are infinitely many types of surface-knots with index equal to three (representing $2$-twist-spun $(2,k)$-torus knots). 

In the paper \cite{CESS06} there are given three pairs of diagrams of classical links such that deforming one of them to the other, requires minimum 2 (or 3 in other cases) Reidemeister III moves. We will give infinitely many diagrams $D$ of a trivial $2$-component link such that deforming $D$ into the trivial diagram with no crossing requires at least $4$ Reidemeister III moves.

\section{Basic definitions}
\

We will work in the smooth category, i.e we will be assuming that all manifolds and functions between them are smooth. An image of an embedding of a closed surface to $\mathbb{R}^4$ is called the \emph{surface-knot}. We will use a word: \emph{classical}, thinking about theory of embeddings of circles $S^1\sqcup\ldots\sqcup S^1\hookrightarrow \mathbb{R}^3$ modulo ambient isotopy in $\mathbb{R}^3$. Two surface-knots are \emph{equivalent} or have the same \emph{type}, if there exists an orientation preserving auto-homeomorphism of $\mathbb{R}^4$, taking one of those surfaces to the other. Without loss of generality we may assume that the image of projection $\pi(x_1, x_2, x_3, t)=(x_1, x_2, x_3)$ is in an \emph{general position}, i.e. the double point set of a surface consists of points whose neighborhood is locally homeomorphic to:
\begin{enumerate}[label={(\roman*)}]
\item two transversely intersecting sheets,
\item three transversely intersecting sheets,
\item the Whitney's umbrella.
\end{enumerate}
Points corresponding to cases (i), (ii), (iii) are called: a \emph{double point}, a \emph{triple point} and a \emph{branch point} respectively, of the projection.

Let $\mathbb{R}^3_t$ denote $\mathbb{R}^3\times\{t\}$ for $t\in\mathbb{R}$.  For a surface $K\subset\mathbb{R}^4$, the family $\{K\cap\mathbb{R}^3_t\}_{t\in\mathbb{R}}$ is called a \emph{motion picture} or simply a \emph{movie} for $K$. Moreover $K\cap\mathbb{R}^3_t$ is a \emph{still} of that movie. Every surface-knot gives us a movie, and from a specific finite number of stills we can recreate completely the type of the corresponding surface-knot.

\begin{prop}[{\cite[p. 12]{CarSai98}}]\label{p01}
In the generic projection of a movie of a surface-knot, a Reidemeister III move on stills corresponds to a triple point of the projection of the surface-knot to $\mathbb{R}^3$.
\end{prop}

More basic terminology and properties may be found in a book \cite{CarSai98}.

\subsection{Twist-spun knots}
\

One of the main families (besides ribbon surface) as objects of study, in surface-knot papers, is that of \emph{twist-spun knots} defined as follows.

\begin{defn}[\cite{Zee65}]
We think of $\mathbb{R}^4$ as an open book decomposition, that is a spun (in the fixed direction) of $\mathbb{R}^3_+$ about $\mathbb{R}^2$. For a classical knot $K$ we take its \emph{tangle} $T$ i.e. a properly embedded arc in $\mathbb{R}^3_+$, with distinct end points $a, b\in\partial\mathbb{R}^3_+$ such that $T\cup[a,b]\subset\mathbb{R}^3$ is a knot of the same type as $K$, where $[a,b]$ is an arc in $\mathbb{R}^2$ connecting $a$ and $b$. Then the geometrical trace of spinning of $T$ about $\mathbb{R}^2$ with additional twisting it in the meanwhile (in the fixed direction) $m$ times in a surrounding ball $B^3$ we call $m$-\emph{twist-spinning} of $K$ and denote it by $\tau^m(K)$.
\end{defn}

Throughout this paper, we do not consider an orientation of a surface-knot but that of ambient space $\mathbb{R}^4$. For a surface-knot $F$, we denote by $F^*$ the one as the mirror reflection of $F$.

\begin{thm}[Zeeman \cite{Zee65}, Litherland \cite{Lit85}]\label{t01}
\
For every classical knot $K$ and $m\in\mathbb{Z}$, we have the property that:

\begin{enumerate}[label={(\roman*)}]
\item $\tau^{-m}(K)=\tau^m(K)^*$.
\item $\tau^{m}(K^*)=\tau^m(K)$.
\end{enumerate}

\end{thm}

\subsection{Hyperbolic splitting and marked diagrams}
\ 

\begin{thm}[Lomonaco \cite{Lom81}, Kawauchi, Shibuya, and Suzuki \cite{KSS82}]

For any surface-knot $F$, there exists a surface-knot $F'$ satisfying the following:
\begin{enumerate}[label={(\roman*)}]
\item $F'$ is equivalent to $F$ and has only finitely many Morse's critical points.
\item All maximal points of $F'$ lie in $\mathbb{R}^3_1$.
\item All minimal points of $F'$ lie in $\mathbb{R}^3_{-1}$.
\item All saddle points of $F'$ lie in $\mathbb{R}^3_0$.
\end{enumerate}

\end{thm}

We call a representation $F'$ in the theorem a \emph{hyperbolic splitting} of $F$. The zero section $\mathbb{R}^3_0\cap F'$ of the hyperbolic splitting gives us a 4-valent graph (with possible loops without vertices). We assign to each vertex a \emph{marker} that informs us about one of the two possible types of saddle points (see Figure \ref{r6}). Making now a projection in general position of this graph to $\mathbb{R}^2$, and imposing crossing types like in the classical knot case, we receive a \emph{marked diagram} (terminology also applies to graphs of that kind which do not come from slicing closed surface).

\begin{figure}[ht]
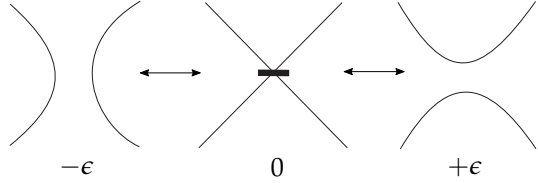

\begin{center}
\begin{lpic}[b(0.5cm)]{./PICTURES/m001(7cm)}
	\lbl[t]{15,-2;$-\epsilon$}
  \lbl[t]{60,-2;$0$}
  \lbl[t]{102,-2;$+\epsilon$}
	\end{lpic}
		\caption{Rules for smoothing a marker.\label{r6}}
\end{center}
\end{figure}

\begin{thm}[Kawauchi, Shibuya, and Suzuki \cite{KSS82}]\label{t04}

Let $F_i$ $(i=1,2)$ be a surface-knot in a hyperbolic splitting, and $D_i$ the marked diagram associeted with the cross-section $F_i\cap\mathbb{R}^3_0$. If $D_1=D_2$, then $F_1$ is equivalent to $F_2$.

\end{thm}

For a marked diagram $D$, we denote by $L_+(D)$ and $L_-(D)$ the diagrams obtained from $D$ by smoothing every vertex as shown in Figure \ref{r6} for $+\epsilon$ and $-\epsilon$, respectively. We have the following characterization of marked diagrams corresponding to surface-knots.

\begin{thm}[Kawauchi, Shibuya, and Suzuki \cite{KSS82}]\label{t05}

Let $D$ be a marked diagram. There exists a surface-knot $F$ in a hyperbolic splitting such that $D$ is associated with the cross-section $F\cap\mathbb{R}^3_0$ if and only if $L_+(D)$ and $L_-(D)$ are diagrams of trivial links in $\mathbb{R}^3$.

\end{thm}

\begin{thm}[Swenton \cite{Swe01}](question asked by Yoshikawa in \cite{Yos94})\ \\

Two surface-knots are equivalent if and only if, their marked diagrams may be transformed one to another by isotopy in $\mathbb{R}^2$ and finite sequence of elementary local moves $Y_1, \ldots,Y_8$ taken from the list in Figure \ref{r1} (their mirror moves and moves having all its markers (in that fragment of a diagram) switched to its second type).
\end{thm}

\begin{figure}[ht]
\begin{center}
\includegraphics[width=.75\textwidth]{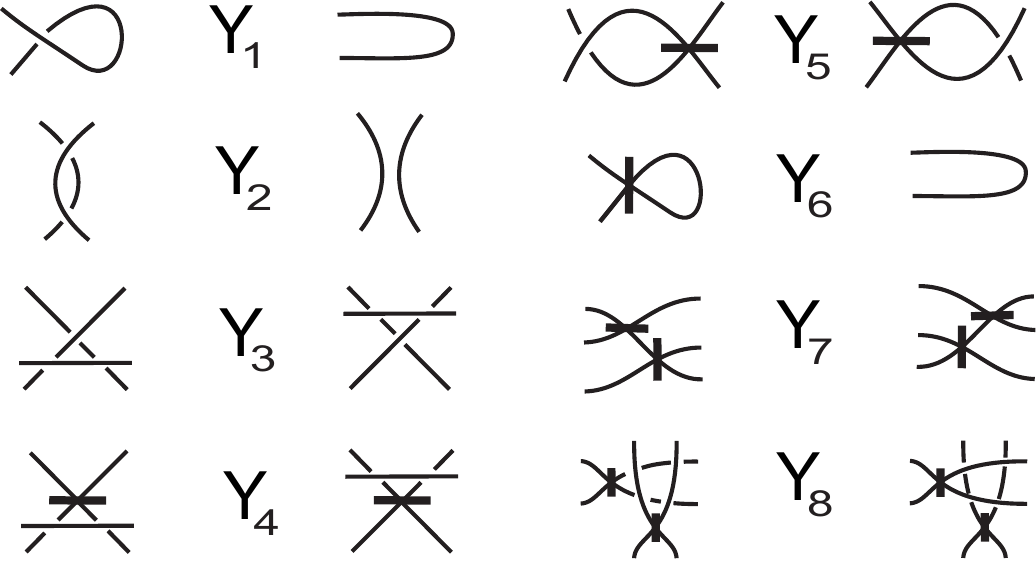}
\caption{\label{r1} Yoshikawa moves}
\end{center}
\end{figure}

\section{Monoid of surface-knots}

\begin{defn}

We say that a marked diagram is in \emph{braid form} if by forgetting about marker types, it is the geometric closure of some singular braid (notion that was developed independently in \cite{Bae92} and \cite{Bir93}) of degree $m$ for some $m\in\mathbb{Z}_+$ (i.e. having $m$ numbered strands).

\end{defn}

\begin{prop}
\ \\
For every surface-knot there exists its marked diagram in braid form.

\end{prop}
\begin{proof}
Forgetting for a moment about markers and leaving singular points at its place, we apply the Alexander's like theorem for singular braids from \cite{Bir93}. Moreover we deduce that it is only required to use moves of type $Y_1, Y_2, \ldots, Y_5$ to do a braid form. Putting now markers back in appropriate manner to vertices, we receive a braid form of marked diagram.

\end{proof}

We now introduce monoid $SSB_m$ consisting of marked diagrams in braid form on $m$ strands. Elements of that monoid, called \emph{singular surface braids} are generated by four types of elements $a_i, b_i, c_i, c_i^{-1}$ for $i=1, \ldots, m-1$, where the correspondence of types of crossings and types of markers between $i$-th and $i+1$-th strands (in the horizontal position) is presented in Figure \ref{r5} (remaining strands in braid are straight lines without crossings).

\begin{figure}[ht]
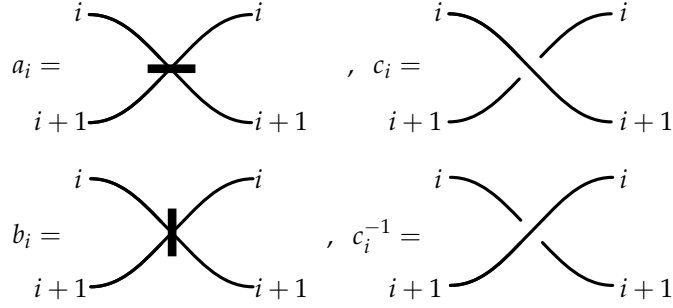

\begin{center}
\begin{lpic}[]{./PICTURES/pic05(7cm)}
	\lbl[r]{-5,42;$a_i=$}
	\lbl[r]{-5,10;$b_i=$}
	\lbl[r]{64,42;$,\;\;c_i=$}
	\lbl[r]{64,10;$,\;\;c_i^{-1}=$}
	\lbl[r]{-1,53;$i$}
	\lbl[r]{-1,21;$i$}
	\lbl[r]{0,32;$i+1$}
	\lbl[r]{0,0;$i+1$}
	\lbl[r]{68,53;$i$}
	\lbl[r]{68,21;$i$}
	\lbl[r]{68,32;$i+1$}
	\lbl[r]{68,0;$i+1$}
	\lbl[l]{32,53;$i$}
	\lbl[l]{32,21;$i$}
	\lbl[l]{32,32;$i+1$}
	\lbl[l]{32,0;$i+1$}
	\lbl[l]{102,53;$i$}
	\lbl[l]{102,21;$i$}
	\lbl[l]{102,32;$i+1$}
	\lbl[l]{102,0;$i+1$}
	\end{lpic}
		\caption{Correspondence of monoid generators.\label{r5}}
\end{center}
\end{figure}

We will indicate our closure of a marked diagram in braid form by adding brackets [] and sometimes adding lower index to it, saying how many strands we are joining.

\begin{exmp}

We have two types of trivially knotted projective planes $[c_1a_1]$ and $[c_1^{-1}a_1]$. Standard torus $\mathbb{T}^2$ can be presented as $[b_1a_1]$.
\end{exmp}

Let the symbol $\Delta_s$ mean (known from braid theory) the positive half-twist in $\mathbb{R}^3$ of $s$ strands involved in the equation we are concerning (it is clear from words indices). It contains only product of generators of type $c_i$.

\begin{defn}\label{d01}
Let $m\in \mathbb{Z}_+$ and $i,k,n\in\{1, \ldots, m-1\}$ such that $|k-i|=1$, moreover let $x_i,y_i\in\{a_i,b_i,c_i, c_i^{-1}\}$.
In monoid $SSB_m$ we introduce following relations.

\begin{enumerate}
\item[A1)] $c_ic_i^{-1}=1$
\item[A2)] $x_iy_n=y_nx_i$ for $n\not=k$
\item[A3)] $x_ic_kc_i=c_kc_ix_k$
\item[A4)] $x_ic_k^{-1}c_i^{-1}=c_k^{-1}c_i^{-1}x_k$
\item[A5)] $a_ib_k=b_ka_i$
\item[A6)] $a_ib_{i-2}(c_{i-1}c_{i-2}c_ic_{i-1})^2=a_ib_{i-2}$ for $i>2$
\item[A7)] $b_ia_{i-2}(c_{i-1}c_{i-2}c_ic_{i-1})^2=b_ia_{i-2}$ for $i>2$
\item[A8)] $a_i^2=a_i$
\item[A9)] $b_i^2=b_i$
\item[A10)] $a_ib_ic_i^2=a_ib_i$
\item[A11)] $a_ib_k\Delta_3=a_ib_k\Delta_3^{-1}$
\end{enumerate}

\end{defn}

Let us denote by $CSB_m$ a subset of $SSB_m$ containing only those elements $x$, that $L_+([x])$ and $L_-([x])$ are diagrams of trivial classical links.

\begin{defn}
We define moreover following Markov type relations (where $n\in \mathbb{Z}_+$, $x_i\in\{a_i,b_i,c_i, c_i^{-1}\}$).

\begin{enumerate}
\item[C1)] $\left[x_iS_n\right]_n=\left[S_nx_i\right]_n$ for $i<n$ and $x_iS_n\in CSB_n$
\item[C2)] $\left[S_n\right]_n=\left[S_nx_n\right]_{n+1}$ for $S_n\in CSB_n$
\end{enumerate}
\end{defn}

\begin{thm}

Making change in algebraic formulation of a surface-knot by using one of relations A1)-A11) or C1)-C2) on words, we receive a formula of equivalent surface-knot.

\end{thm}
\begin{proof}
The algebraic relations A1)-A10), C1)-C2) were deduced from topological local moves preserving the type of surface-knot from paper \cite{Yos94}. The proof of relation A11) for $k=i+1$ is given as in Figure \ref{r0014} (for $k=i-1$ it may be done by analogy, changing marker types and using relation A5)).

\begin{figure}[ht]
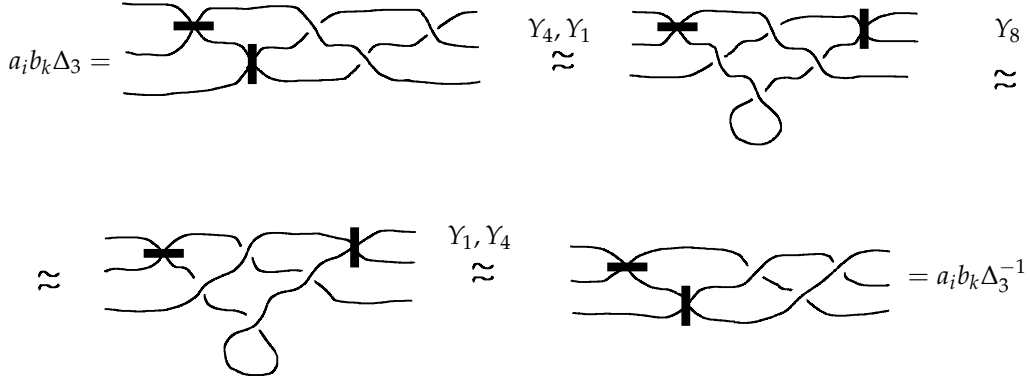

\begin{center}
\begin{lpic}[b(1cm)]{./PICTURES/MJ_57(13cm)}
	\lbl[r]{11,47;$a_ib_k\Delta_3=$}
	\lbl[b]{78,50;$Y_4, Y_1$}
	\lbl[b]{145,50;$Y_8$}
	\lbl[b]{66,19;$Y_1, Y_4$}
	\lbl[l]{130,15;$=a_ib_k\Delta_3^{-1}$}
	\end{lpic}
	\caption{A11) move\label{r0014}}
\end{center}
\end{figure}

\end{proof}

\begin{remark}
It is still an open problem, whether any pair of marked diagrams in braid form of equivalent surface-knot, can be transformed one another by using only relations A1)-A11) and C1)-C2).
\end{remark}

\begin{lem} 
\
Under the assumptions of Definition \ref{d01} the following relation holds:
\begin{enumerate}
\item[A12)] $x_i\Delta_n=\Delta_n x_{n-i}$ for $1\leq i<n\leq m$.
\end{enumerate}
\end{lem}
\begin{proof}

It follows directly from geometric observation after making positive half-twist in $\mathbb{R}^3$ of first $n$ strands with element $x_i$.

\end{proof}

\begin{prop}[J. H. Przytycki] 
\
Under the assumptions of Definition \ref{d01} the following relation holds:
\begin{enumerate}
\item[A13)] $a_ic_k^{-1}b_ic_k\Delta_3^2=a_ic_k^{-1}b_ic_k$
\end{enumerate}
\end{prop}
\begin{proof}
$a_ic_k^{-1}b_ic_k\Delta_3^2\overset{\text{A12)}}{=}a_ic_k^{-1}\Delta_3^2b_ic_k=a_ic_k^{-1}(c_ic_kc_i)^2b_ic_k\overset{\text{A3)}}{=}a_ic_k^{-1}c_kc_ic_kc_ic_kc_ib_ic_k\overset{\text{A1), A2)}}{=}a_ic_ic_kc_ic_kb_ic_ic_k\overset{\text{A2), A3)}}{=}c_ia_i\Delta_3b_ic_ic_k\overset{\text{A12)}}{=}c_ia_ib_k\Delta_3c_ic_k\overset{\text{A11)}}{=}c_ia_ib_k(\Delta_3)^{-1}c_ic_k\overset{\text{A2)}}{=}a_ic_ic_i^{-1}c_k^{-1}c_i^{-1}c_ib_ic_k\overset{\text{A1)}}{=}a_ic_k^{-1}b_ic_k$

\end{proof}

Using the diagram given in Montesinos' paper \cite{Mon86}, we can write down in terms of our monoid, the formula for every twist-spun knots as follows.

\begin{prop}\label{p04}
\ \\
Let a classical knot $\widehat{K}$ be the plat closure of the braid $K$ on $2m+1$ strands, then $$\tau^n(\widehat{K})=\left[\left(\prod_{i=1}^{m}a_{2i}\right)K\left(\prod_{i=1}^{m}b_{2i}\right)K^{-1}\Delta_{2m+1}^{2n}\right].$$
\end{prop}

\begin{exmp}

Let us now see how to unknot the $n$-twist-spun trivial knot on $3$ strands. We have

$\tau^n(\widehat{c_1^{-1}})=[a_2c_1^{-1}b_2c_1\Delta^{2n}]_3\overset{\text{A13)}}{=}[a_2c_1^{-1}b_2c_1]\overset{\text{C1)}}{=}[c_1a_2(c_1)^{-1}b_2]\overset{\text{A1)}}{=}[c_2^{-1}c^2c_1a_2c_1^{-1}b_2]\overset{\text{A3)}}{=}[c_2^{-1}a_1c_2c_1c_1^{-1}b_2]\overset{\text{A1)}}{=}[c_2^{-1}a_1c_2b_2]\overset{\text{C1)}}{=}[c_2b_2c_2^{-1}a_1]_3\overset{\text{C2)}}{=}[c_2b_2c_2^{-1}]_2\overset{\text{A2), A1)}}{=}[b_2]_2\overset{\text{C2)}}{=}[1]_1.$

\end{exmp}

\begin{exmp}

We now use this algebraic method to show proof of Theorem \ref{t01}(i).

Let $K=\widehat{R}$, then
\ \\
$\tau^{-n}(K)=\left[\left(\prod_{i=1}^{m}a_{2i}\right)R\left(\prod_{i=1}^{m}b_{2i}\right)R^{-1}\Delta_{2m+1}^{-2n}\right]\overset{\text{mirror reflection}}{=}\\=\left[\Delta_{2m+1}^{2n}R\left(\prod_{i=1}^{m}b_{2i}\right)R^{-1}\left(\prod_{i=1}^{m}a_{2i}\right)\right]^*\overset{\text{relation C1)}}{=}\\=\left[\left(\prod_{i=1}^{m}a_{2i}\right)\Delta_{2m+1}^{2n}R\left(\prod_{i=1}^{m}b_{2i}\right)R^{-1}\right]^*\overset{\text{relation A12)}}{=}\\=\left[\left(\prod_{i=1}^{m}a_{2i}\right)R\left(\prod_{i=1}^{m}b_{2i}\right)R^{-1}\Delta_{2m+1}^{2n}\right]^*=\tau^{n}(K)^*$.

\end{exmp}

\subsection{Index of a surface singular braid}

\begin{defn}
The \emph{singular braid index} of a surface-knot $F$, denoted by $Ind_S(F)$, is the minimum degree among all surface singular braids, that its closure gives a marked diagram of a surface equivalent to $F$. 
\end{defn}

It is easily seen that if $Ind_S(F)=1$ (i.e. $F=[1]_1$) then $F$ is the standard unknotted $2$-sphere $\mathbb{S}^2$. We will investigate this notion further.

\begin{thm}
If $Ind_S(F)=2$ then there are exactly six types of surface-knots $F$. Moreover, there exist infinitely many surface-knot types $F$ such that $Ind_S(F)=3$.
\end{thm}

\begin{proof}
Let us consider elements of $CSB_2$. From the relation A2) it follows that all of them are commutative. So each surface is in the form $[a_1^{\alpha}b_1^{\beta}c_1^{\gamma}c_1^{-\delta}]_2$, for $\alpha, \beta, \gamma, \delta\geq 0$. By relations A8) and A9) it follows that all these surface-knots are in the form $[a_1^{\alpha}b_1^{\beta}c_1^{\gamma}c_1^{-\delta}]_2$, where $\alpha, \beta \in\{0,1\}$ and $\gamma, \delta\geq 0$.

If $\alpha\cdot\beta=1$ then by the relation A10) we have $\gamma, \delta \in\{0,1\}$ giving us two types of surfaces: $[a_1b_1]$ and $[a_1b_1c_1]$.
If $\alpha\cdot\beta=0$ then one of the resolutions $L_+$ or $L_-$ is a diagram of the torus link $T(2,\delta-\gamma)$. This classical link must be trivial by Theorem \ref{t05}, so we have that $|\delta-\gamma|\in\{0,1\}$ and using the relation A1) if needed, we have that $\gamma, \delta \in\{0,1\}$. This gives us four more types of surfaces: $[1]_2, [c_1], [a_1c_1], [a_1c_1^{-1}]$. 

All of those above mention six types of surface links are known to be mutually distinct.

As for the case $Ind_S(F)=3$, we can take a family of the surface knots $\tau^2(T(2,k))$ for odd prime $k$. They are mutually distinct (see paper \cite{AsaSat05}) and they are $3$-strand closure of the surface singular braid word $a_2c_1^{-k}b_2c_1^k\Delta^{4}$.

\end{proof}

\subsection{Minimal number of Reidemeister III moves}
\ \\
We now give a family of pairs of diagrams of classical links such that deforming one of them to the other requires minimum four Reidemeister III moves.

\begin{thm}

There exists a family of classical diagrams $D_{n,k}$ for $n\geq 2$ and odd $k\geq 3$ of $2$-component links with $6n+2k$ crossings such that there is at least four Reidemeister III moves required to transform the diagram $D_{n,k}$ into the trivial diagram without any crossings. 
\end{thm}

\begin{proof}
Let us consider a diagram $D_{n,k}$ as a (modified) plat closure of word $c_1^k\Delta^{2n}c_1^{-k}$, presented in Figure \ref{r9}.

\begin{figure}[ht]
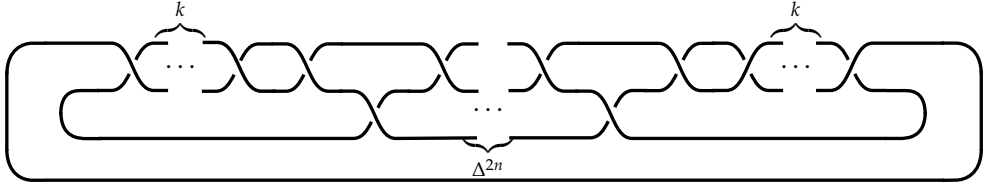

\begin{center}
\begin{lpic}[b(1cm),t(.3cm)]{./PICTURES/MJ_98(13cm)}
  \lbl[b]{37,30;$\overbrace{\text{}\quad\text{}}^{k}$}
	\lbl[b]{37,24;$\ldots$}
	\lbl[b]{167,30;$\overbrace{\text{}\quad\text{}}^{k}$}
	\lbl[b]{167,24;$\ldots$}
	\lbl[t]{102,9.7;$\underbrace{\text{}\quad\text{}}_{\Delta^{2n}}$}
	\lbl[b]{102,15;$\ldots$}
	\end{lpic}
	\caption{Diagram $D_{n, k}$\label{r9}}
\end{center}
\end{figure}

From Proposition \ref{p01} we know that every Reidemeister III move in a motion picture stills corresponds to one triple point in a surface diagram. By uniqueness of surface-knot type from a given marked diagram (Theorem \ref{t04}) it follows that, it is sufficient to prove that there exists a surface-knot $F$, such that $L_-(F)$ is the $D_{n,k}$ diagram and $L_+(F)$ can be unlinked without any Reidemeister III move; finally that the surface $F$ has at least four triple points in every projection to $\mathbb{R}^3$.

The latter one follows from combining theorems of Satoh from paper \cite{Sat05b} and Cochran from paper \cite{Coc83}, for $F$ being the $n$-twist-spun torus knot $T(2,k)$ for $n\geq 2$ and odd integer $k\geq 3$.

To prove the ability of unlink $L_+(F)$ with only using Reidemeister I-st or II-nd moves, we will proceed directly. The diagram $L_+(F)$ is at the beginning a (modified) plat closure of the braid word $c_1^{-k}c_1^{k}\Delta^{2n}$, after obvious reduction of the word $c_1^{-k}c_1^{k}$, we sequentially reduce every plat closure (from one side) of expression $\Delta^2$ as in Figure \ref{r22}. We receive at the end a diagram of two disjoint circles on the plane.

\begin{figure}[ht]
\begin{center}
\includegraphics[width=.95\textwidth]{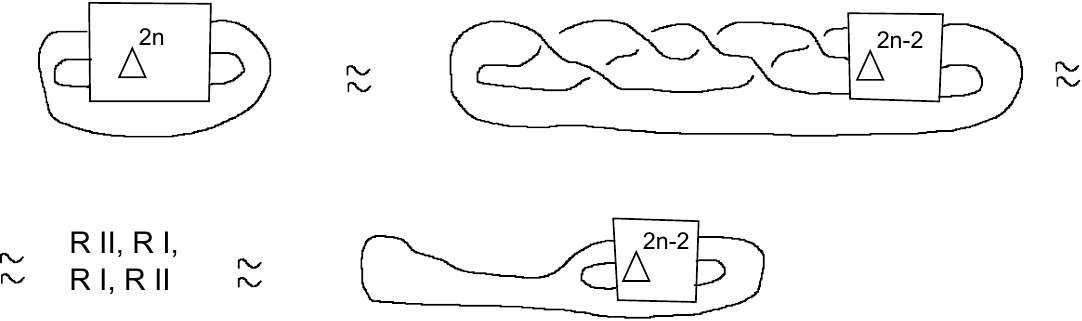}
\caption{\label{r22} part of reduction of $L_+(F)$}
\end{center}
\end{figure}

Finally we see that $L_-(F)$ is a (modified by planar isotopy) presented in Figure \ref{r9} diagram $D_{n,k}$, because we have that $\tau^n(T(2,k))=[a_2c_1^{-k}b_2c_1^{k}\Delta^{2n}]$.
\end{proof}

The definition of twist-spun torus knots raises an interesting question about duality of its parameters. For example $\tau^n(T(2,k))$ and $\tau^k(T(2,n))$ for $n=1$ or $k=1$ gives the same (unknotted) surface knot. But in the case $k, n$ are distinct odd integers greater than one, quandle cocycle invariants do not distinguish them (see paper \cite{AsaSat05}). In our algebraic language we state the following.

\begin{ques}
For what odd different integers $k, n>1$ we have $[a_2c_1^{-k}b_2c_1^{k}\Delta^{2n}]=[a_2c_1^{-n}b_2c_1^{n}\Delta^{2k}]?$
\end{ques}

\subsection*{Acknowledgments}

Research of M. Jab\l onowski was partially supported by grant BW 5107-5-0343-0. The author would like to thank the referee for carefully reading the paper. 

\end{document}